\documentclass[10pt]{amsart}
\usepackage[a4paper]{geometry}
\usepackage{multicol}
\usepackage[utf8]{inputenc}
\usepackage[T1]{fontenc}
\usepackage{hyperref}
\usepackage{latexsym}
\usepackage{amsmath}
\usepackage{amscd}
\usepackage{amsxtra}
\usepackage{amsfonts}
\usepackage{amssymb}

\usepackage{graphicx}
\graphicspath{{figures/}} 
\usepackage[width=.9\linewidth]{caption}
\theoremstyle{plain}  
\newtheorem{theorem}{Theorem}[section]

\newtheorem*{theorem*}{Theorem}

\newtheorem{proposition}[theorem]{Proposition}

\theoremstyle{definition}

\theoremstyle{remark}

\newtheorem{remark}[theorem]{Remark}
\newtheorem*{remark*}{Remark}

\newtheorem*{claim*}{Claim}

\numberwithin{equation}{section}

\newcommand{\suchthat}{\;\;|\;\;}

\newcommand{\abs}[1]{\lvert#1\rvert}

\renewcommand{\leq}{\leqslant}

\renewcommand{\geq}{\geqslant}

\renewcommand{\setminus}{\smallsetminus}
\newcommand{\into}{\hookrightarrow}

\newcommand{\RR}{\mathbb{R}}

\newcommand{\ZZ}{\mathbb{Z}}
\newcommand{\CC}{\mathbb{C}}

\newcommand{\EE}{\mathbb{E}}
\newcommand{\HH}{\mathbb{H}}

\newcommand{\PSL}{\mathrm{PSL}}

\newcommand{\SU}{\mathrm{SU}}

\newcommand{\SL}{\mathrm{SL}}
\newcommand{\tSL}{\widetilde{\mathrm{SL}}}

\newcommand{\SO}{\mathrm{SO}}
\newcommand{\Sp}{\mathrm{Sp}}

\DeclareMathOperator{\Hom}{Hom}

\DeclareMathOperator{\Isom}{Isom}

\newcommand{\spmat}[4]{
  \left(
    \begin{smallmatrix}
      #1 & #2 \\
      #3 & #4
    \end{smallmatrix}
  \right)
}

\begin{document}

\title{Geometry on Surfaces and Higgs bundles}
\author{Peter B. Gothen}
\email{pbgothen@fc.up.pt}

\date{5 October, 2023}
\address{Centro de Matemática da Universidade do Porto and Departamento de Matemática, Faculdade de Ciências da Universidade do Porto, Portugal}

\thanks{Partially supported by CMUP under the projects UIDB/00144/2020,
  UIDP/00144/2020, and the project EXPL/MAT-PUR/1162/2021 funded by FCT (Portugal)
  with national funds.}

\begin{abstract}
There are three complete plane geometries of constant curvature:
spherical, Euclidean and hyperbolic geometry. We
explain how a closed oriented surface can carry a geometry which
locally looks like one of these. Focussing on the hyperbolic case we
describe how to obtain all hyperbolic structures on a given
topological surface, and how to parametrise them. Finally we introduce
Higgs bundles and explain how they relate to hyperbolic surfaces.
\end{abstract}

\maketitle

\begin{multicols}{2}
\section{Introduction}

The idea of considering geometry on a surface is well known to
inhabitants of Planet Earth. Indeed, as any explorer knows, spherical
geometry is appropriate. In this geometry distance is measured along
arcs of great circles. These are the \emph{geodesics} of spherical
geometry, just like the geodesics of plane Euclidean geometry are
segments of straight lines.

The spherical surface and the Euclidean plane are both
\emph{complete}, meaning that any geodesic can be extended
indefinitely. Moreover they both have constant curvature,  positive in the
case of the sphere, and zero  in the case of the plane.
There is also a complete $2$-dimensional geometry of
constant negative curvature, namely the hyperbolic plane (which we
shall introduce below).

The sphere is an example of a \emph{closed surface}, i.e., a surface
which is compact and has no boundary (as opposed to a closed disk, for
example). Topologically, closed orientable surfaces are classified by
the genus $g$, a non-negative integer: a surface of genus $g$ can be
realised inside $3$-space as a ``$g$-holed torus''  as illustrated in
Figure~\ref{fig:surfaces}.
 We have seen that the genus zero surface supports spherical geometry
but what about the other surfaces? Our first main goal in this article
is to explain how the torus (genus one) supports a complete geometry
which locally looks like the Euclidean plane, while a surface of genus
$g\geq 2$ can be given a complete locally hyperbolic geometry. This
involves considering certain special subgroups of the matrix group
$\SL(2,\RR)$.  We shall then  see how the algebra and geometry of the matrix
group $\SL(2,\RR)$ interact in interesting ways, and how this sheds
light on the question of which subgroups give rise to hyperbolic  surfaces.
\begin{center}
  \includegraphics[width=0.9\linewidth]{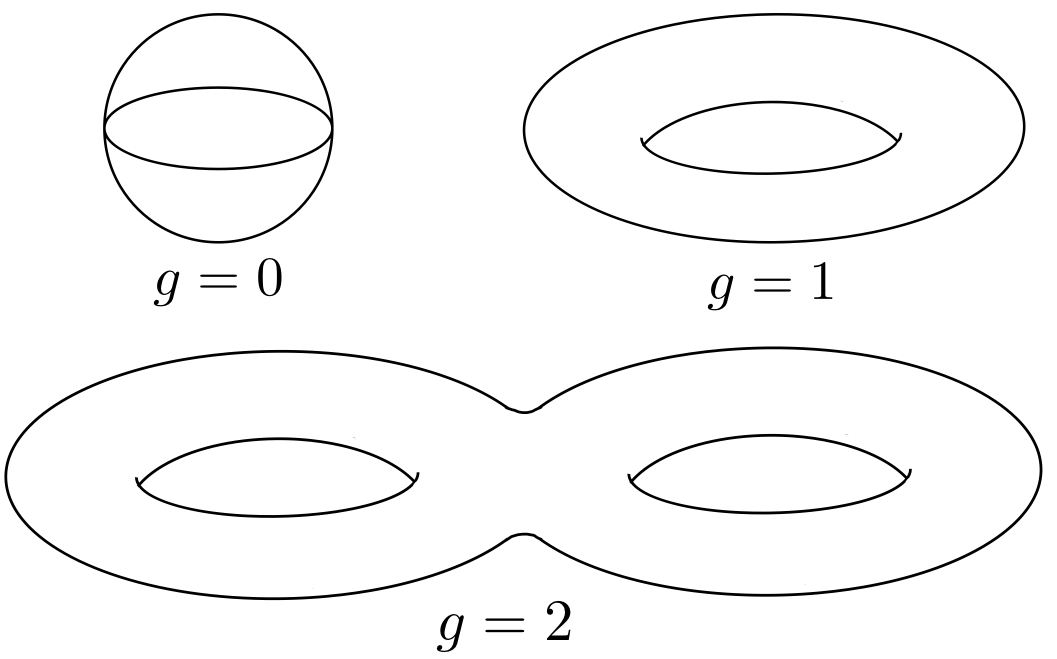}
  \captionof{figure}{Genus of a surface}
   \label{fig:surfaces}
\end{center}

Our second main goal is to explain how considering the set of all possible
hyperbolic structures on a fixed topological surface of genus
$g\geq 2$ leads to interesting and beautiful mathematics. Thus we
introduce moduli spaces and explore some of their properties.

Finally, we shall give an introduction to Higgs bundles. We shall show
how they can be used to shed new light on the theme of hyperbolic 
structures on surfaces and indicate their role in recent
generalisations of some of the results explained earlier in the
article.

The paper is mostly expository, only the final
Section~\ref{sec:cayley} includes some results in which the author has
been involved.

For reasons of space the references are by no means complete, but we
hope the interested reader will be able to use them as a starting
point for further exploration.

\section{Euclidean surfaces}

We want to explain how to do Euclidean geometry on a closed surface,
in a way which makes the generalisation to the hyperbolic case natural.

\subsection{The Euclidean plane}
\label{sec:euclidean-plan}
Using Cartesian coordinates we identify the Euclidean plane $\EE^2$
with the coordinate plane  $\RR^2$. Distance is determined  by calculating the length of a
parametrised curve $\alpha\colon[a,b]\to \RR^2$ in the usual way:
$l(\alpha) = \int_a^b\abs{\alpha'(t)} dt$. This is usually expressed
by saying that in Cartesian coordinates $(x,y)\in\RR^2$ on $\EE^2$ the
Euclidean element  of arc length $ds$ is given by
\begin{equation}
  \label{length}
  ds^2 = dx^2 + dy^2.
\end{equation}
Moreover, the distance preserving transformations form
a group, $\Isom(\EE^2)$, which is called the isometry group of
$\EE^2$. An example of isometries are \emph{translations}.
Using coordinates $\EE^2\cong\RR^2$, the translation
$A\colon\EE^2\to\EE^2$ by the vector $\mathbf{a}\in\RR^2$ can be written
\begin{displaymath}
  A(P) = P+ \mathbf{a}.
\end{displaymath}

\subsection{Euclidean surfaces}
\label{sec:euclidean-surfaces}
The Euclidean plane $\EE^2$ is obviously not a closed
surface. However, we can build a closed surface by taking a parallellogram in
the plane and gluing its opposite sides, as illustrated in
Figure~\ref{fig:torus}: sides labeled with the same letter are to be
glued with the orientation indicated by the arrows.
We can carry out this process in $3$-space --- in a non-distance preserving way! ---
to convince ourselves that
the resulting surface is really a topological torus.
\begin{center}
  \includegraphics[width=0.9\linewidth]{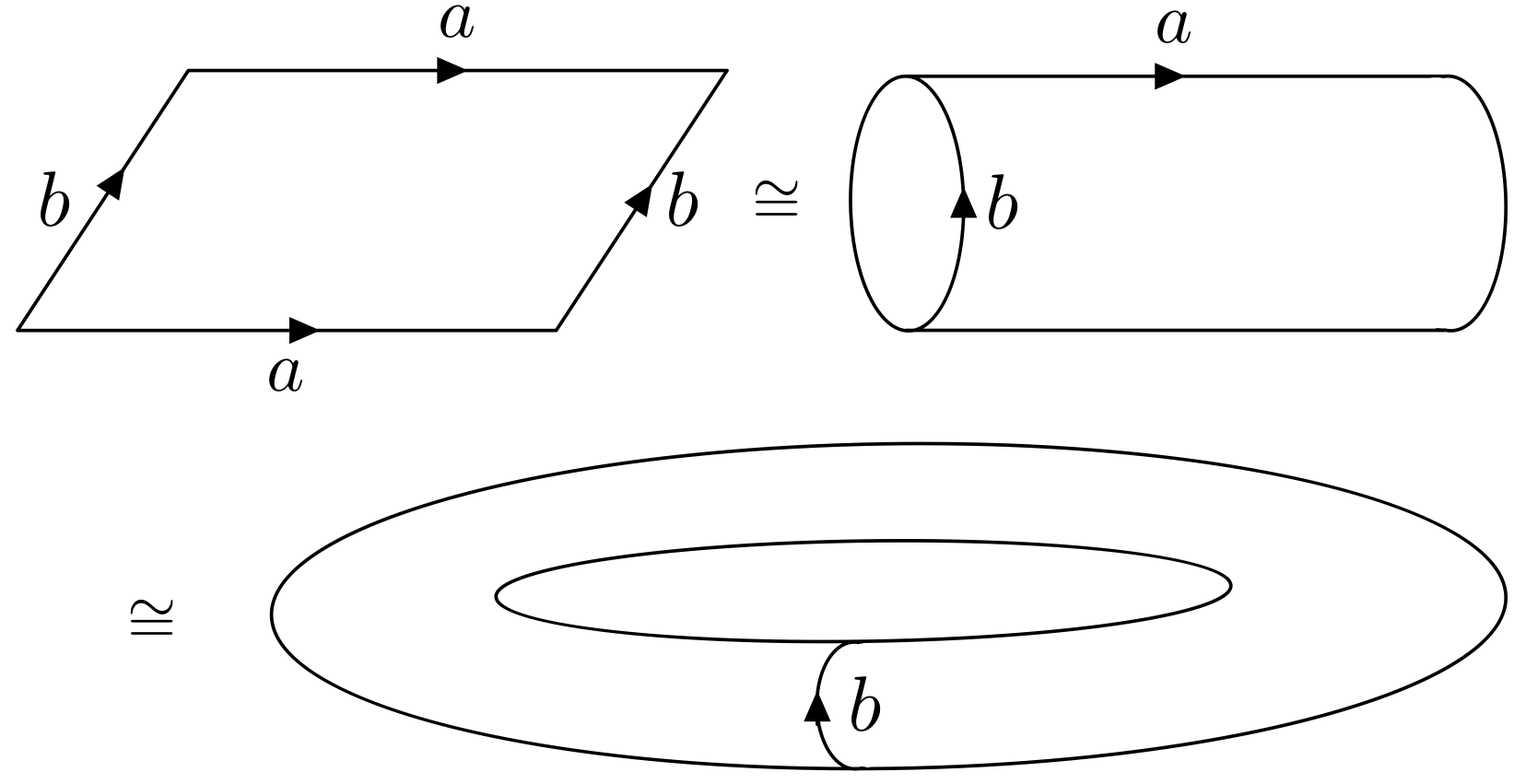}
  \captionof{figure}{A Euclidean surface}
  \label{fig:torus}
\end{center}

More formally, we take linearly independent vectors $\mathbf{a}$ and
$\mathbf{b}$ generating the sides labeled $a$ and $b$. Let $A$ and
$B$ be the translations by the vectors $\mathbf{a}$ and $\mathbf{b}$,
respectively, and consider the subgroup
\begin{displaymath}
 \Gamma = \langle A,B \rangle \subseteq \mathrm{Isom}(\EE^2)
\end{displaymath}
generated by them inside the isometry group of $\EE^2$. This is just
the group of translations by vectors of the form
$n\mathbf{a}+m\mathbf{b}$, where $m,n\in\ZZ$. Since translations
commute, the generators of $\Gamma$ satisfy the single relation
\begin{displaymath}
  [A,B] = I,
\end{displaymath}
where $[A,B] := ABA^{-1}B^{-1}$ is the
\emph{commutator} and $I$ is the identity.
The \emph{orbit space}
\begin{displaymath}
  \EE^2 / \Gamma
\end{displaymath}
is obtained identifying points $P,Q\in\EE^2$ if there is a
$\gamma\in\Gamma$ such that $Q=\gamma(P)$. Its points correspond to \emph{orbits} $\Gamma\cdot Q = \{\gamma(Q)\suchthat\gamma\in\Gamma\}$.
Thus each point of the
interior of the paralellogram generated by $\mathbf{a}$ and
$\mathbf{b}$ corresponds to a unique point of $\EE^2 / \Gamma$, and
pairs of points on opposite sides are identified via the corresponding
translation, thus realising the desired gluing. Hence
$\EE^2 / \Gamma$ is a locally Euclidean surface, which is
topologically a torus.

As illustrated in Figure~\ref{fig:vertex} the four vertices of the
parallelogram $P$ get identified in $\EE^2 /
\Gamma$. Moreover, there are four copies of the parallelogram meeting there, which
fit together because of the relation $[A,B] = I$; the Euclidean
metric is not distorted because the sum of the internal angles of the
parallelogram is exactly $2\pi$.
\begin{center}
  \includegraphics[width=0.9\linewidth]{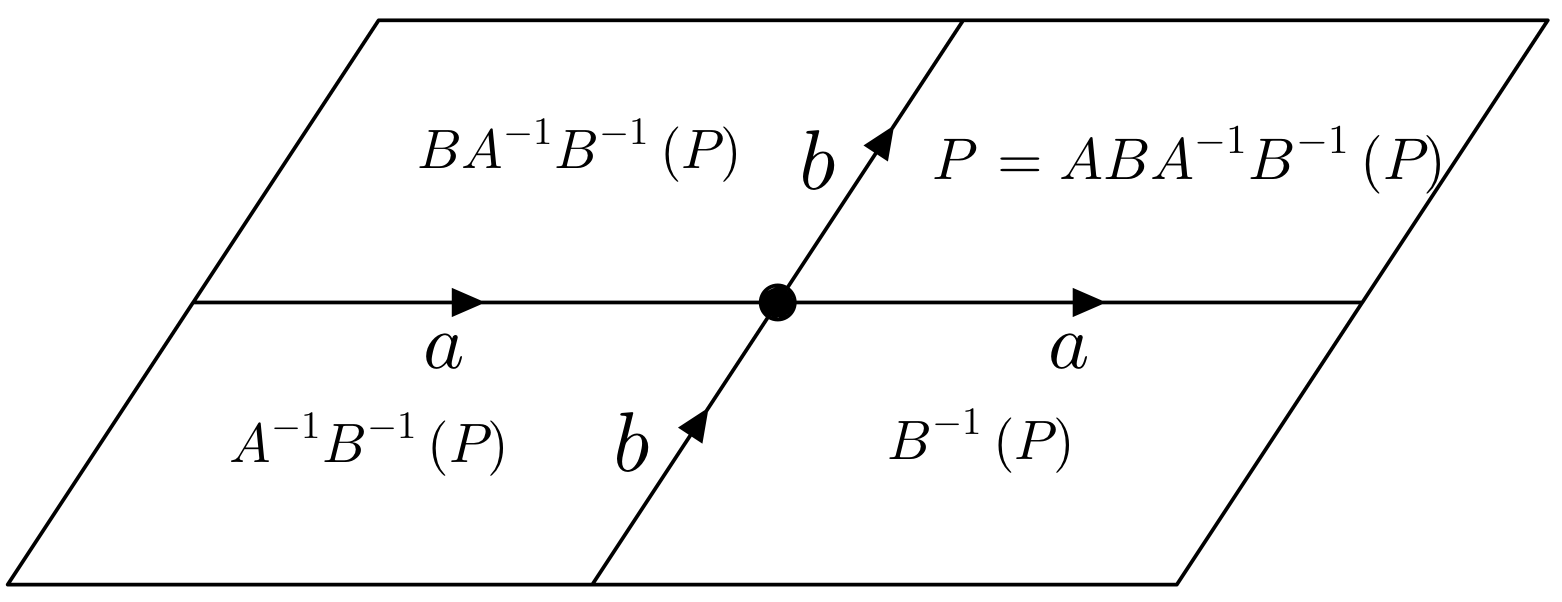}
  \captionof{figure}{Four copies of $P$}
  \label{fig:vertex}
\end{center}

From a more abstract point of view, a key point is that the group $\Gamma$ has the following
property: for every point $P$ in $\EE^2$, there is an open
neighbourhood $U$ such that $\gamma(U)\cap U = \emptyset$ for all
$\gamma$ different from the identity.  We say the action of $\Gamma$
on $\EE^2$ is \emph{properly discontinuous}. This property ensures
that for each $Q\in U$ its orbit
$\Gamma\cdot Q = \{\gamma(Q)\suchthat\gamma\in\Gamma\}$ has a unique
representative (namely $Q$ itself) in $U$, so that $U$ works as a
\emph{coordinate patch} for $\EE/\Gamma$ around $\Gamma\cdot P$.
This, together with the fact that the elements of $\Gamma$ are isometries, means that
$\EE^2 /\Gamma$ has a well defined distance function: indeed, the arc
length of a parametrised curve in $\EE^2 /\Gamma$ can be calculated
using the formula (\ref{length}) which is invariant under isometries.


Note that we can also construct non-compact surfaces in this way. For
example, if we take $\Gamma$ to be the subgroup generated by a single
translation, we obtain a cylinder. This is a locally Euclidean surface
which, unlike the Euclidean torus, can be easily visualised in
$3$-space by rolling up a sheet of paper.

The so-called \emph{Killing--Hopf Theorem} implies that any complete connected 
locally Euclidean surface can be represented as $\EE^2/\Gamma$, where
$\Gamma$ acts freely and properly discontinuously on $\EE^2$ (see, for
example, Stillwell~\cite{stillwell:1992}).

\section{Hyperbolic surfaces}

\subsection{The hyperbolic plane}
We start by describing the \emph{hyperbolic plane} $\HH^2$.
Hyperbolic geometry is different from spherical and Euclidean geometry
in that it is not possible to embed (smoothly) $\HH^2$ in Euclidean
$3$-space in a distance preserving way. Instead
we consider
%
%
the \emph{upper half plane model},
defined by
\begin{displaymath}
  \HH^2 = \{z=x+iy\in\CC\suchthat y>0\}
\end{displaymath}
with element of arc length
\begin{displaymath}
  ds^2 = \frac{dx^2+dy^2}{y^2}.
\end{displaymath}
In the model $\HH^2$ geodesics are open arcs of semi-circles
orthogonal to the real axis $\RR=\{y=0\}\subset\CC$ together with open
half-lines orthogonal to $\RR$. Note that the hyperbolic plane is
complete, so these curves do in fact have infinite hyperbolic length.
Moreover, orientation preserving isometries can be represented by
\emph{M\"obius transformations}
\begin{displaymath}
  z \mapsto A\cdot z = \frac{az+b}{cz+d},
\end{displaymath}
where
\begin{displaymath}
  A=
  \begin{pmatrix}
    a & b \\ c & d
  \end{pmatrix}
   \in \SL(2,\RR)
\end{displaymath}
is a real $2\times 2$-matrix of determinant one. As examples we can
take
$A=
\left(
\begin{smallmatrix}
  \rho & 0 \\
  0 & \rho
\end{smallmatrix}
\right)
$
which gives a hyperbolic translation whose axis is the imaginary axis
in $\HH^2$, and
$A=
\left(
\begin{smallmatrix}
  \cos\theta & -\sin\theta \\
  \sin\theta & \cos\theta
\end{smallmatrix}
\right)
$
which gives a hyperbolic rotation about $i\in\HH^2$ by the angle $2\theta$.

We note that $A$ and
$-A$ define the same Möbius transformation, so the group of
orientation preserving isometries is really the quotient group
$\PSL(2,\RR) = \SL(2,\RR) / \{\pm I\}$.
We shall mostly ignore this distinction in what follows but it will
become relevant in Section~\ref{sec:moduli-space-reps} below.

We finish this section by commenting on the topology of
$\SL(2,\RR)$. Identifying the set of all $2\times 2$-matrices
$\left(
    \begin{smallmatrix}
      a & b\\
      c & d
    \end{smallmatrix}\right)$
  with $\RR^4$, the group $\SL(2,\RR)$ is the subset cut out by the
  equation $ac-bd=1$. Thus it has a topology inherited from
  $\RR^4$. In fact the Implicit Function Theorem applied to this
  equation shows that $\SL(2,\RR)$ is a $3$-dimensional \emph{Lie group}, meaning
  that it can be covered by local coordinate systems in $3$-space and
  that the group operations are differentiable in these
  coordinates.


\subsection{Hyperbolic surfaces}

As we shall see, a closed orientable topological surface of genus $g$
can be given a hyperbolic structure for any $g\geq 2$. In the case of
$g=2$, take an octagon with gluing instructions to create a surface  as illustrated in Figure~\ref{fig:octagon} . If we cut the
octagon along the diameter indicated, we see that indeed the resulting
surface has genus $2$, as desired.
\begin{center}
  \includegraphics[width=0.9\linewidth]{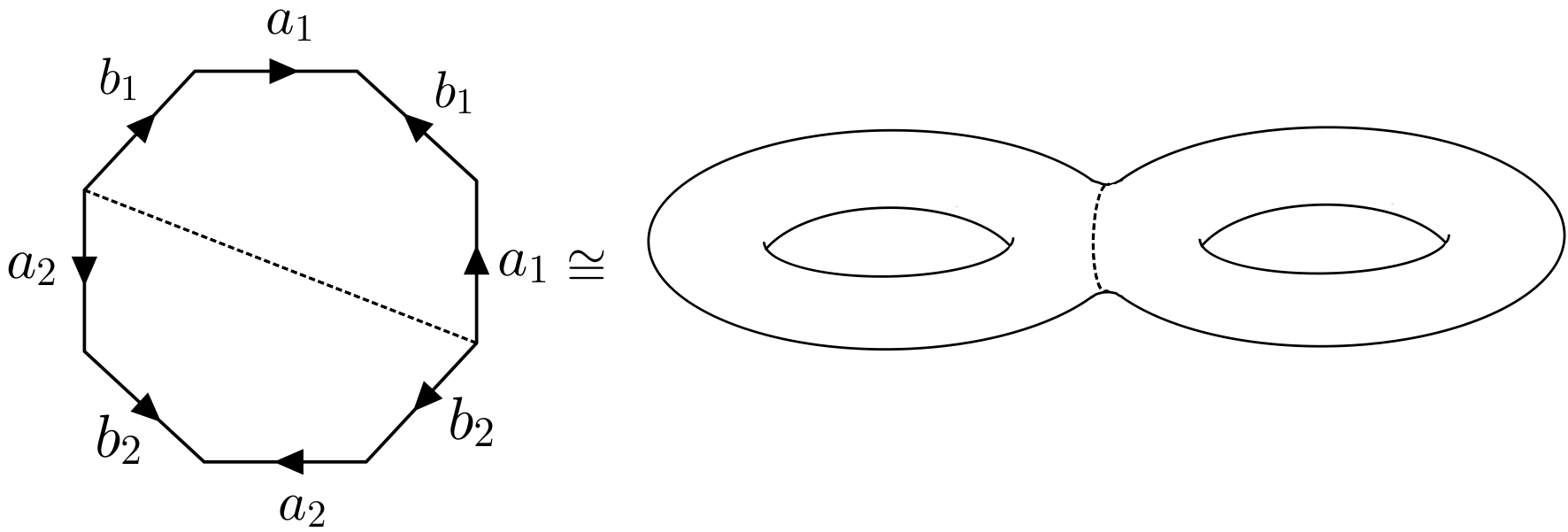}
  \captionof{figure}{Genus 2 surface from an octagon}
  \label{fig:octagon}
\end{center}

In order to get a nice
\emph{hyperbolic} surface, the octagon should be taken in the hyperbolic
plane (it will look very different from that of Figure~\ref{fig:octagon}). And, in a manner analogous to the Euclidean case, we require that pairs of
sides which are to be glued have the same length. Moreover, the
vertices of the octagon all get identified to one point in the
surface, so the internal angles should add up to $2\pi$. This
condition sounds strange to our Euclidean wired brains but, it is a
fact that such an octagon exists.\footnote{In fact, in hyperbolic geometry the
  sum of the internal angles of a polygon depends on its area!}

In order to write the surface as $\HH^2/\Gamma$ for a suitable
subgroup $\Gamma\subset\SL(2,\RR)$ we take hyperbolic translations
$A_i$ and $B_i$ giving the required identifications, and let
$\Gamma$ be the group generated by these translations. The octagon
then becomes a fundamental domain for the action of
$\Gamma$. The condition that the interior angles add up to $2\pi$ is
equivalent to the identity
\begin{displaymath}
  [A_1,B_1][A_2,B_2] = I
\end{displaymath}
in $\SL(2,\RR)$.
In general, we let $\Gamma_g$ be the abstract group
\begin{displaymath}
  \Gamma_g = \langle a_1,b_1,\dots,a_g,b_g \suchthat \prod_{i=1}^g
  [a_i,b_i] = 1 \rangle.
\end{displaymath}
This group is known as a \emph{surface group}.\footnote{The group
  $\Gamma_g$ can be identified wth the fundamental group of a
  topological surface of genus $g$.}
In view of the genus $2$ example it is hopefully not a surprise that
genus $g$ surfaces can be obtained from subgroups of $\SL(2,\RR)$
isomorphic to $\Gamma_g$. In order to study all such subgroups we
consider homomorphisms $\rho\colon\Gamma_g\to \SL(2,\RR)$ (often
also called \emph{representations}). We say that $\rho$ is
\emph{Fuchsian} if it is injective and its image is discrete, i.e.,
consists of isolated points.\footnote{Recall that topological notions
  make sense viewing $\SL(2,\RR)\subseteq\RR^4$.}
When $\rho$ is Fuchsian it can be proved that the action of $\Gamma_g$
 on $\HH^2$ is properly discontinuous. Hence the orbit space 
\begin{displaymath}
  S_\rho := \HH^2 / \rho(\Gamma_g),
\end{displaymath}
is a nice hyperbolic surface of genus $g$ with charts coming from $\HH^2$. 
Conversely, the Killing--Hopf Theorem again tells us that any closed
orientable 
hyperbolic surface is of this form.\footnote{As
  already noted, we should really consider representations  to
  $\PSL(2,\RR)$. However, it turns out that representations defining
  closed hyperbolic  surfaces can always be lifted to $\SL(2,\RR)$.}

However, it is certainly not true that any homomorphism
$\rho\colon\Gamma_g\to\SL(2,\RR)$ defines a closed hyperbolic
surface: for example, the trivial homorphism clearly does not!
This leaves us with the following

\textbf{Question:} Let $\rho\colon\Gamma_g\to\SL(2,\RR)$ be a representation. How
can we tell if $\rho$ defines a closed hyperbolic
surface?

\section{Topology and algebra of $\SL(2,\RR)$}

In order to answer the question at the end of the last section we
shall define an invariant of representations
$\rho\colon\Gamma\to\SL(2,\RR)$. For that we shall need to understand
how the topology and algebra of $\SL(2,\RR)$ interact.

The subgroup $\SO(2)\subseteq \SL(2,\RR)$ of rotation matrices
\begin{math}
  E(\theta) =\left(
  \begin{smallmatrix}
    \cos\theta & -\sin\theta\\
    \sin\theta & \cos\theta
  \end{smallmatrix}\right)
\end{math}
can be identified with a circle.
The map $E\colon\RR\to\SO(2)$, $\theta\mapsto E(\theta)$ wraps the
real line around the circle, and it
satisfies $E(0) = I$ and
$E(\theta_1+\theta_2) = E(\theta_1)E(\theta_2)$. In other words, $E$
is a group homomorphism from the additive group $\RR$ to $\SO(2)$.

Now,
thinking of $\SO(2)$ inside $\SL(2,\RR)$, we want to extend this
picture and find a group $\tSL(2,\RR)$ containing $\RR$, with a
surjective group homomorphism $p\colon\tSL(2,\RR)\to\SL(2,\RR)$ which
restricts to $E\colon\RR\to\SO(2)$, i.e.,
making
the diagram
\begin{displaymath}
  \begin{CD}
    \RR @>>> \tSL(2,\RR) \\
    @VV{E}V @VV{p}V \\
    \SO(2) @>>> \SL(2,\RR)
  \end{CD}
\end{displaymath}
commutative (the horisontal maps are inclusions).
In fact it follows
from general theory that such a group exists and is essentially
unique; it is known as the \emph{universal covering group} of
$\SL(2,\RR)$. We shall explain how it can be constructed explicitly,
following \cite[\S1.8]{lion-vergne:1980}, using the action of
$\SL(2,\RR)$ on the hyperbolic plane $\HH^2$.

So let $A=
\spmat{a}{b}{c}{d}
\in\SL(2,\RR)$. Write
\begin{displaymath}
  j(A,z) = cz+d
\end{displaymath}
for the denominator of
$A\cdot z$.
Note that 
\begin{displaymath}
  j(E(\theta),i) = i\sin\theta+\cos\theta  = e^{i\theta},
\end{displaymath}
which indicates that this function can used to keep track of the
phase $\theta$.
For each fixed $A$, we can consider the holomorphic function
\begin{align*}
  \HH^2 &\to \CC\setminus\{0\}\\
  z &\mapsto j(A,z) = cz+d.
\end{align*}
Observe that $cz+d\neq 0$ for $z\in\HH$.
Therefore, since $\HH^2$ is simply
connected, there is a continuous determination of the logarithm of
$j(A,z) = cz+d$, i.e., a continuous map
\begin{math}
  \phi\colon \HH^2 \to \CC
\end{math}
such that
\begin{displaymath}
  e^{\phi(z)} = cz+d.
\end{displaymath}
We want to make the point that such a $\phi$ can be explicitly calculated: simply choose a value $\theta$ for the
argument $\arg(ci+d)$, write $ci+d=re^{i\theta}$ and let
$\phi(i)=\log(r)+{i\theta}$. Then 
\begin{displaymath}
  \phi(z) - \phi(i) = \int_{\gamma}\frac{dz}{z}
  =\int_0^1\frac{c(z-i)dt}{c(i+t(z-i))+d}
\end{displaymath}
(here $\gamma$ parametrises the segment joining $ci+d$ to $cz+d$). 
Note that $\phi$ is not unique, but it is uniquely determined by the
choice of $\phi(i)$. Thus any two determinations $\phi$
differ by an integer multiple of $2\pi i$.

Now {define} $\tSL(2,\RR)$ as the set of pairs
\begin{math}
  (A,\phi),
\end{math}
where $A\in\SL(2,\RR)$ and $\phi\colon\HH^2\to\CC$ is any  continuous determination of the logarithm
of $j(A,z)$. 
The product on $\tSL(2,\RR)$ is defined by
\begin{displaymath}
  (A_1,\phi_1)\cdot(A_2,\phi_2) = (A_1A_2,\tilde{\phi}),
\end{displaymath}
where
\begin{displaymath}
  \tilde{\phi}(z) := \phi_1(A_2\cdot z) + \phi_2(z).
\end{displaymath}
It is an easy calculation to check that
\begin{displaymath}
  j(A_1A_2,z) = j(A_1,A_2\cdot z)j(A_2,z)
\end{displaymath}
which implies that indeed
\begin{displaymath}
  e^{\tilde{\phi}(z)} = j(A_1A_2, z)
\end{displaymath}
as required. It is not hard to check that this defines a group
structure on $\tSL(2,\RR)$. For example, for $A=I$, the identity matrix, we can
take $\phi(z)=0$ and $(A,0)$ is the neutral element. Moreover,
\begin{math}
  (A,\phi)^{-1} = (A^{-1},\tilde{\phi}),
\end{math}
where
\begin{equation}
  \label{eq:phi-inverse}
  \tilde{\phi}(z) = -\phi(A^{-1}\cdot z).
\end{equation}

The projection $p\colon\tSL(2,\RR)\to\SL(2,\RR)$ is of course just
$(A,\phi)\mapsto A$. The inclusion $\RR\into\tSL(2,\RR)$ is given by
\begin{math}
  \theta\mapsto(E(\theta),\phi_\theta),
\end{math}
where $\phi_\theta$ is the determination of $\log(j(E(\theta),z))$
which satisfies $\phi_\theta(i) = i\theta$ (recall that
$j(E(\theta),i) = e^{i\theta}$).

\begin{proposition}
  \label{prop:kernel-p}
  The kernel of $p\colon\tSL(2,\RR)\to\SL(2,\RR)$ consists of pairs
  $(I,\phi)$, where $I$ is the
  identity matrix and $\phi$ is a constant function taking values in
  $2\pi\ZZ\subset\RR$.
\end{proposition}

\begin{proof}
  Clearly $p(A,\phi) = I$ if and only if $A=I$. Moreover,
  $j(I,z) = 1$, so $\phi$ is a determination of the logarithm of the
  constant function $z\mapsto 1\in C$, i.e., it is a constant
  $\phi\in 2\pi\ZZ\subset\RR$.
\end{proof}

\section{The Toledo Invariant}
\label{sec:toledo-invariant}

Let $\rho\colon\Gamma\to\SL(2,\RR)$ be a representation. We shall
associate an integer invariant to $\rho$. This invariant is known as
the Toledo invariant, even though it was actually introduced by
Milnor~\cite{milnor:1957}, and sometimes is referred to as the Euler number.
Write
\begin{displaymath}
  A_i = \rho(a_i),\quad B_i=\rho(b_i)
\end{displaymath}
for $i=1,\dots,g$. Choose \emph{lifts} $\tilde{A}_i$ and $\tilde{B}_i$ in
$\tSL(2,\RR)$ such that $p(\tilde{A}_i)=A_i$ and $p(\tilde{B}_i)=B_i$, and define the \emph{Toledo invariant} of $\rho$ to be
\begin{displaymath}
  \tau(\rho) = \frac{1}{\pi}\prod_{i=1}^g[\tilde{A}_i,\tilde{B}_i].
\end{displaymath}
In view of the relation defining $\Gamma_g$, the product
  $\prod_{i=1}^g[\tilde{A}_i,\tilde{B}_i]$ is in the kernel of
  $p$. Hence Proposition~\ref{prop:kernel-p} shows that {the
    Toledo invariant is an even integer}.\footnote{Odd Toledo
    invariants arise from representations $\rho\colon\Gamma_g\to\PSL(2,\RR)$ which do
    not lift to $\SL(2,\RR)$.}

From the description of $\tSL(2,\RR)$ of the preceding section,
  it is easy to check that the Toledo invariant is well defined, i.e.,
  that it does not depend on the choice of lifts: the main point is
  that the ambiguity in the choice of $\phi$ is canceled by
  (\ref{eq:phi-inverse}), because each lift occurs together with its
  inverse in the commutator. Moreover, the
  Toledo invariant of a representation  defined by matrices $A_i$ and $B_i$
 {can be explicitly calculated}.

\section{Goldman's theorem}

A celebrated inequality due to Milnor \cite{milnor:1957} states that
\begin{displaymath}
  \abs{\tau(\rho)} \leq 2g-2
\end{displaymath}
for every representation $\rho\colon\Gamma_g\to\SL(2,\RR)$. The
following beautiful result shows that representations with maximal
Toledo invariant (known as \emph{maximal representations}) have a special
geometric significance.

\begin{theorem}[Goldman \cite{goldman:1985}]
  \label{thm:goldman-max}
  A representation $\rho\colon\Gamma_g\to\SL(2,\RR)$ is Fuchsian if
  and only if $\abs{\tau(\rho)} = 2g-2$. 
\end{theorem}

\begin{remark}
  One might wonder about the significance of the sign of the Toledo
  invariant. If we conjugate a representation $\rho$ by the outer
  automorphism of $\SL(2,\RR)$ given by conjugation by a reflection we
  obtain a representation $\bar{\rho}$ with
  $\tau(\bar{\rho}) = -\tau(\rho)$. In fact, the hyperbolic surface
  $S_{\bar{\rho}}$ is obtained from $S_{\rho}$ by a change of
  orientation, i.e., by composing all charts with a reflection in
  $\HH^2$.
\end{remark}

\section{The moduli space of representations}
\label{sec:moduli-space-reps}

Let us now take a global view and consider all representations of
$\Gamma_g$ in $\SL(2,\RR)$ simultaneously.
The \emph{representation space} for representations of $\Gamma_g$ in
$\SL(2,\RR)$ is the set of homomorphisms
\begin{math}
  \Hom(\Gamma_g,\SL(2,\RR)).
\end{math}
It is natural to consider $\rho_1$ and $\rho_2$ equivalent if they
differ by overall conjugation by an element of $\SL(2,\RR)$,
corresponding to a change of basis in $\RR^2$.
It also turns out that two hyperbolic structures on the same topological surface are isometric by an isometry which can be
continuously deformed to the identity if and only if the corresponding
Fuchsian representations are equivalent in this sense.
Thus the \emph{moduli
  space} of representations is defined to be the orbit space 
\begin{displaymath}
  \mathcal{R}(\Gamma_g,\SL(2,\RR))
  = \Hom(\Gamma_g,\SL(2,\RR)) / \SL(2,\RR)
\end{displaymath}
under the conjugation action.\footnote{In order to get a Hausdorff
  quotient, one should in fact exclude representations whose action on
$\RR^2$ is not semisimple.}

A homomorphism $\rho\colon\Gamma_g\to\SL(2,\RR)$ is determined by 
$2g$ matrices
\begin{displaymath}
    A_i = \rho(a_i),\; B_i=\rho(b_i),\quad i=1,\dots,g
\end{displaymath}
satisfying the single relation $\prod[A_i,B_i] = I$.  Hence
$\Hom(\Gamma_g,\SL(2,\RR))$ can be identified with the subspace of
$\RR^{6g}$ cut out by the $3$ scalar equations given by
$\prod[A_i,B_i] = I$ (the equation takes values in the $3$-dimensional
group $\SL(2,\RR)$). It follows that it is
a variety of dimension $6g-3$.
The conjugation action by $\SL(2,\RR)$ reduces the
dimension by $3$, and so the moduli space has dimension 
\begin{displaymath}
  \dim \mathcal{R}(\Gamma_g,\SL(2,\RR)) = 6g-6.
\end{displaymath}

The Toledo invariant separates the moduli space into subspaces
\begin{displaymath}
  \mathcal{R}_d \subseteq \mathcal{R}(\Gamma_g,\SL(2,\RR))
\end{displaymath}
corresponding to representations with invariant
$d$. Goldman~\cite{goldman:1988} showed that the
$\mathcal{R}_d$ are in fact connected components of the moduli space,
except in the maximal case $\abs{d}=2g-2$. It turns out that
$\mathcal{R}_{2g-2}$ has $2^{2g}$ connected components. However, these
components get identified after
projecting onto
\begin{displaymath}
  \mathcal{R}(\Gamma_g,\PSL(2,\RR)),
\end{displaymath}
which thus has just one connected component with Toledo invariant
$2g-2$. This is not
surprising because, by Goldman's
Theorem~\ref{thm:goldman-max}, the subspace  $\mathcal{R}_{2g-2}$ is
exactly the locus of
Fuchsian representations and, moreover, any two Fuchsian representations
 into $\SL(2,\RR)$ define the same hyperbolic surface if and only if
they coincide after projecting to $\PSL(2,\RR)$.  Accordingly, the
corresponding connected component $\mathcal{T} \subseteq \mathcal{R}(\Gamma_g,\PSL(2,\RR))$ is known as the \emph{Fuchsian locus}.
As we have seen, it parametrises all
hyperbolic structures on the topological surface $S_g$ up to a natural
equivalence. It is a classical result that the space of such hyperbolic
structures can be identified with $\RR^{6g-6}$. In the next section we shall explain how a
parametrisation of this space can be obtained using Higgs bundles.




\section{Higgs bundles}
\label{sec:higgs-bundles}

We now describe how the results of the previous section can be
understood using non-abelian Hodge theory, a subject founded by
Hitchin \cite{hitchin:1987a} and Simpson \cite{simpson:1992}.

\subsection{Riemann surfaces and holomorphic bundles}
\label{sec:riemann-surfaces}

A \emph{Riemann surface} $X$ is a topological surface together with a
family of local charts which together cover the surface,
and are such that changes of coordinates are holomorphic functions between
open sets in $\CC$. An example of this is the \emph{Riemann sphere}
$\hat\CC = \CC \cup \infty$: we use the standard coordinate $z$ in
$\CC$ and around $\infty\in\hat\CC$ we use the coordinate $w =
1/z$. Thus the change of coordinates $T\colon\CC\setminus\{0\}\to
\CC\setminus\{0\}$ given by $w=T(z)=1/z$ is holomorphic in the domain
where both $z$ and $w$ are defined.

In particular, if we have a hyperbolic surface $S_g \cong \HH^2 / \Gamma$ for a Fuchsian representation of $\Gamma$,
then the local coordinates in $\HH^2$ give $S_g$ the structure of a
Riemann surface: indeed the changes of coordinates are Möbius
transformations of $\HH^2$, which are certainly
holomorphic.
We write $X_\rho$ for the Riemann surface constructed from a Fuchsian
representation $\rho$ in this way.

Note that not all holomorphic maps define isometries of $\HH^2$, so the concept
of Riemann surface is less rigid than that of hyperbolic
surface. However, the
famous Uniformisation Theorem, due to Köbe and Poincaré, asserts that
any Riemann surface can be represented as a hyperbolic surface. This
means, in particular, that the space of all Riemann surfaces with the
same underlying topological surface of genus $g$
(up to a suitable equivalence) can be identified with the Fuchsian
locus $\mathcal{T}$. When thought of in this way, it is
known as \emph{Teichmüller space}.

A \emph{holomorphic line bundle} $L\to X$ on a Riemann surface $X$ is a
holomorphic family of $1$-dimensional complex vector spaces
parametrised by $X$.  Thus, for each $p\in X$ we have a
$1$-dimensional complex vector space $L_p$, which varies
holomorphically with $p$. The simplest example is the \emph{trivial
  bundle} $L=X\times\CC\to X$; here the map is projection onto $X$
and the fibre $L_p=\{p\}\times\CC$ for $p\in X$ with its vector space
structure coming from $\CC$. Locally on $X$, a
holomorphic line bundle is required to look like the product
$U\times\CC$, where $U\subseteq X$ is open.  We say that $L$ is
trivialised over $U$.  This means that a holomorphic line bundle can
be given by an open covering $\{U_{\alpha}\}$ of $X$ and 
\emph{trivialisations}
\begin{displaymath}
L_{|U_{\alpha}}\cong U_{\alpha}\times\CC
\end{displaymath}
for each $\alpha$. This gives rise to 
\emph{transition functions}
\begin{displaymath}
g_{\alpha\beta}\colon U_{\alpha}\cap U_{\beta}\to\CC\setminus\{0\}
\end{displaymath}
which compare the isomorphisms $L_p\cong\CC$ given by the
trivialisations over $U_{\alpha}$ and $U_{\beta}$, respectively.

More important
than the line bundles themselves are their \emph{sections}. These are
holomorphic maps $s\colon X\to L$ such that $s(p)\in L_p$ for all
$p\in X$. A section of the trivial bundle $U\times\CC$ over $U$ is of
course nothing but a map $s\colon U\to\CC$, and if we have
local trivialisations of a line bundle $L$ as above, a holomorphic
section $s$ corresponds to a collection of holomorphic maps
$s_\alpha\colon U_{\alpha}\to\CC$ which glue correctly, i.e., satisfy
the condition
\begin{displaymath}
  s_\alpha(p) = g_{\alpha\beta}(p)s_\beta(p)
\end{displaymath}
for $p\in U_{\alpha}\cap U_{\beta}$.
As an illustrative example, we take the \emph{canonical bundle}
$K\to X$. Its sections are \emph{holomorphic differentials}. In a
local coordinate $z$ on $X$ a holomorphic differential, say $\alpha$, can be written
\begin{displaymath}
  g(z)dz
\end{displaymath}
for a holomorphic function $g(z)$ and if $h(w)dw$ is the representation of
$\alpha$ in another holomorphic coordinate $w=T(z)$, then
\begin{displaymath}
  g(z)dz = h(T(z))d(T(z)) = h(T(z))T'(z)dz.
\end{displaymath}
Thus a holomorphic differential can be represented by a collection of
holomorphic functions locally defined on coordinate charts which
transform according to the preceding rule. It turns out that the
vector space of holomorphic differentials on a closed Riemann surface
$X$ of genus $g$, usually denoted by
$H^0(X,K)$, is finite dimensional, of dimension $2g-2$. More generally,
the vector space $H^0(X,L)$ of holomorphic sections of any holomorphic line
bundle $L\to X$ is finite dimensional. Any holomorphic line bundle has
a topological invariant called its \emph{degree}; in case $L$ has a non-zero
holomorphic section, this is the number of zeroes of such a section,
counted with multiplicity. For example, the degree of the canonical
bundle is $2g-2$. The fact that this is the same as the dimension of
$H^0(X,K)$ is a consequence of a fundamental result known as the
Riemann--Roch formula.

We can perform the usual operations of linear algebra, like
taking duals and tensor products, fibrewise on line bundles. Thus, if
$L$ and $M$ are line bundles with transition functions
$g_{\alpha\beta}$ and $h_{\alpha\beta}$, respectively, the tensor
product $L\otimes M$ has transition functions
$g_{\alpha\beta}h_{\alpha\beta}$ (pointwise multiplication in
$\CC\setminus\{0\}$) and the dual bundle $L^*$ has transition functions
$g_{\alpha\beta}^{-1}$.



\subsection{Higgs bundles}

A \emph{$\PSL(2,\RR)$-Higgs bundle} on $X$ consists of three pieces of
data
\begin{displaymath}
  (L,\beta,\gamma)
\end{displaymath}
where, $L\to X$ is a holomorphic line bundle, and
$\beta\in H^0(X;K\otimes L)$ and $\gamma\in H^0(X;K\otimes L^*)$ can
be seen as holomorphic differentials which take values in the line
bundles $L$ and $L^*$, respectively.

In a manner analogous to the conjugation
action on representations, there is a natural notion of isomorphism of
Higgs bundles, and the set of isomorphism classes of
$\PSL(2,\RR)$-Higgs bundles forms the \emph{moduli
  space} $\mathcal{M}(X,\PSL(2,\RR))$. It is a complex algebraic
variety of complex dimension $3g-3$. We note that in order to get a
reasonable moduli space it is necessary to restrict to so-called
semistable Higgs bundles. This is analogous to the way in which one
restricts to semisimple representations in the moduli space of
representations.

The \emph{Non-abelian Hodge Theorem} (due to Corlette, Donaldson,
Hitchin and Simpson) for this situation states the following.

\begin{theorem}
  There is a real analytic isomorphism
  \begin{displaymath}
    \mathcal{R}(\Gamma_g,\PSL(2,\RR))
    \cong \mathcal{M}(X,\PSL(2,\RR))
  \end{displaymath}
\end{theorem}

This is a remarkable theorem for many reasons. Here we just point out
that while the character variety $\mathcal{R}$ is \textbf{real} and
depends only on the topological surface of genus $g$ (through its
fundamental group), the moduli space $\mathcal{M}$ depends on the
Riemann surface structure $X$ given to the topological surface and has a \textbf{complex} structure.

For fixed $d$ we denote by $\mathcal{M}_d$ the subspace of
$\PSL(2,\RR)$-Higgs bundles $(L,\beta,\gamma)$ with $\deg(L)=d$. Then
we have 
$\mathcal{R}_d\cong\mathcal{M}_d$ under
the non-abelian Hodge Theorem. In particular, the Fuchsian locus
$\mathcal{T}$ corresponds to $\mathcal{M}_{2g-2}$.

\subsection{Hitchin's parametrisation of $\mathcal{T}$}

A particular class of $\PSL(2,\RR)$-Higgs bundles can be obtained by
taking $L=K$. Then $\gamma$ is a section of the line bundle $K\otimes
K^*$ which is naturally isomorphic to the trivial line bundle on
$X$. In other words, $\gamma$ is simply a holomorphic function on $X$,
so we can set $\gamma=1$ (the constant function). Moreover, $\beta$ is
a section of $K^2=K\otimes K$. In other words it is a \emph{quadratic
  differential}, so it can locally be written as $\beta(z) =
b(z)(dz)^2$, where $b(z)$ satisfies an appropriate transformation rule under
changes of coordinates. The vector space $H^0(X,K^2)$ of quadratic
differentials on $X$ has complex dimension $3g-3$ which equals the
dimension of the moduli space $\mathcal{M}(X,\PSL(2,\RR))$.
This construction defines a map
\begin{align*}
  \Psi\colon H^0(X,K^2) &\to \mathcal{M}(X,\PSL(2,\RR)),\\
  \beta&\mapsto (K,\beta,1).
\end{align*}
The semistability condition alluded to earlier implies that all Higgs
bundles in $\mathcal{M}_{2g-2}$ arise in this way. Hence $\Psi$ is an
isomorphism onto its image $\mathcal{M}_{2g-2}$.

From the non-abelian Hodge Theorem we already knew that
$\mathcal{M}_{2g-2}\cong\mathcal{T}$ is a connected component. But
the Higgs bundle construction gives an alternative proof. Using gauge
theoretic methods Hitchin
also shows that $\mathcal{M}_{2g-2}$ parametrises all 
hyperbolic metrics on the topological surface underlying
$X$. Moreover, under
this parametrisation the Higgs bundle $(K,1,0)$ corresponds to the
hyperbolic metric which uniformises $X$. Thus
Hitchin's approach gives alternative proofs of Goldman's theorems and
the Uniformisation Theorem.



\subsection{The general Cayley correspondence}
\label{sec:cayley}

Hitchin \cite{hitchin:1992} generalised the construction of the map
$\Psi$ to a map
\begin{displaymath}
  \Psi\colon \bigoplus_{i} H^0(X,K^{d_i}) \to \mathcal{M}(X,G)
\end{displaymath}
whose image is again a connected component of the \emph{moduli space
$\mathcal{M}(X,G)$ of $G$-Higgs bundles} for any simple split real Lie
group $G$, nowadays known as a \emph{Hitchin component}.\footnote{In the case of classical matrix groups this means
  that $G$ is one of the groups $\SL(n,\RR)$, $\Sp(2n,\RR)$, $\SO(p,p)$ and
  $\SO(p,p+1)$.}
The domain of $\Psi$ is a direct sum of spaces of higher holomorphic
differentials on $X$; the integers $d_i$ are determined by the Lie
group $G$ (in fact they are the exponents of its Lie algebra).

Similar constructions of special connected components have later been
given for Hermitian groups $G$ of non-compact tube type, such as
$\SU(p,p)$ (see, for example,
\cite{bradlow-garcia-prada-gothen:2003,bradlow-garcia-prada-gothen:hss-higgs,biquard-garcia-rubio:2017}). In
this case the domain of the map $\Psi$ turns out to be a moduli space
$\mathcal{M}_{K^2}(X,G')$ of so-called $K^2$-twisted $G'$-Higgs
bundles, for a certain real Lie group $G'$ associated to $G$ (known as
its \emph{Cayley partner}).

Both Hitchin components and Cayley components are special because
they are not (as all other known components of the moduli space)
detected by standard topological invariants of the underlying bundles
and the Higgs fields satisfy a certain non-degeneracy condition.

Recently (see \cite{sopq-IM:2019,BCGGO:2021} and the recent survey
\cite{bradlow:2023}) both of these constructions have been unified and
generalised.  The class of Lie groups $G$ covered are characterised by
the fact that their Lie algebras admit a \emph{magical
  $\mathfrak{sl}_2$-triple}. This new Lie theoretic notion builds on
ideas of Hitchin~\cite{hitchin:1992} and generalises
that of a principal $\mathfrak{sl}_2$-triple introduced by Kostant.
Conjecturally the generalised Cayley components obtained by this
construction account for all ``special'' (in the sense of the previous
paragraph) connected components of
the moduli space and thus opens the door to a complete determination
of this important topological invariant.

One important piece of supporting evidence for this conjecture comes
from the area of \emph{Higher Teichmüller Theory}.  Higher Teichmüller
theory has developed in parallel with the Higgs bundle story just
described, and there has been a rich cross-fertilisation of ideas
between the two areas.  We cannot do justice to this fast-growing, rich and
important area of mathematics here but refer the interested reader to
\cite{guichard-wienhard:2018,wienhard:2018}
and references therein. Very briefly, a higher Teichmüller space is a
connected component of the moduli space of representations, which
consists exclusively of discrete and injective representations, like
the Fuchsian locus in the $\PSL(2,\RR)$-case. It turns out that the
generalised Cayley components are indeed higher Teichmüller spaces
\cite{BCGGO:2021,guichard-labourie-wienhard:2021}, and it is expected
that all higher Teichmüller spaces are thus obtained.


\providecommand{\bysame}{\leavevmode\hbox to3em{\hrulefill}\thinspace}

\end{multicols}

\end{document}